\documentclass[11pt,leqno]{amsart}

\makeatletter
\@namedef{subjclassname@2020}{\textup{2020} Mathematics Subject Classification}
\makeatother

\usepackage[top=2.5 cm,bottom=2 cm,left=2 cm,right=3 cm]{geometry}
\usepackage[utf8]{inputenc}
  \usepackage{color}

\usepackage{esint}
\usepackage{graphicx}
\usepackage{hyperref}
\usepackage{bm}

\numberwithin{equation}{section}

\newtheorem{example}{Example}[section]
\newtheorem{definition}[example]{Definition}
\newtheorem{theorem}[example]{Theorem}
 
\newtheorem{lemma}[example]{Lemma}
 
\newtheorem{remark}[example]{Remark}
\newtheorem*{maintheorem*}{Main Theorem}
\allowdisplaybreaks
\numberwithin{equation}{section}

\renewcommand{\i}{\ifmmode\mathit{\mathchar"7010 }\else\char"10 \fi}
\renewcommand{\j}{\ifmmode\mathit{\mathchar"7011 }\else\char"11 \fi}
\newcommand{\R}{\mathbb{R}}
\newcommand{\N}{\mathbb{N}}

\DeclareMathOperator*{\sign}{sign}
\newcommand{\sgn}[1]{\sign\left(#1\right)}
\newcommand{\abs}[1]{\left|#1\right|}
\newcommand{\norm}[1]{\left\|#1\right\|}

\newcommand{\vfi}{\varphi}

\newcommand{\px}{\partial_x}
\newcommand{\pt}{\partial_t}

\newcommand{\ptx}{\partial^2_{xt}}

\newcommand{\pxx}{\partial^2_{xx}}

\newcommand{\eps}{\varepsilon}

\newcommand{\ue}{u_{\eps}}
\newcommand{\uek}{u_{\eps_k}}

\newcommand{\re}{\rho_{\eps}}
\newcommand{\rek}{\rho_{\eps_k}}

\newcommand{\me}{m_{\eps}}
\newcommand{\pe}{p_{\eps}}
\newcommand{\pek}{p_{\eps_k}}

\newcommand{\wpe}{{\widetilde p_\eps}}
\newcommand{\wre}{{\widetilde \rho_\eps}}

\newenvironment{Definitions}
{%

\begin{enumerate}}%
{\end{enumerate}}

\newenvironment{DefinitionsE}
{%

\begin{enumerate}}%
{\end{enumerate}}

\begin{document}

\title[Power Plants]{Analysis of an asymptotic {thermo fluid-dynamic model} for parabolic trough power plants}

\author[Coclite]{G. M. Coclite}
\author[Gasser]{I. Gasser}

\address[Giuseppe Maria Coclite]{\newline Dipartimento di Meccanica,
  Matematica e Management, \newline Politecnico di Bari, \newline Via
  E. Orabona 4 -- 70125 Bari, Italy}
\email[]{giuseppemaria.coclite@poliba.it}

\address[Ingenuin Gasser]{\newline
  Fachbereich Mathematik, \newline Universit\"at Hamburg, \newline
  Bundesstrasse 55 --20146 Hamburg, Germany}
\email[]{ingenuin.gasser@uni-hamburg.de}

\date{\today}




\begin{abstract}
	Parabolic trough power plants transform solar radiative energy into thermal energy which is then typically used to produce electricity.
	We consider a model derived in \cite{BGSP} to describe parabolic trough power plants. In particular, the thermo-fluid dynamics is studied in a single collector pipe where the solar radiation is concentrated. The model is the result of simplifying assumptions and asymptotic processes on the underlying mass, momentum and energy balance equations. { We show existence of solutions for the model. In addition we study the long-time behaviour and the stationary problem.} 
\end{abstract}

\maketitle

%
%
\section{Introduction}
\label{sec:intro}
There is a world wide strongly increasing demand on energy, in particular energy  based on renewable sources. Reasons for that are the necessity to reduce the carbon dioxide production because of the evident climate change, the necessity to reduce the air pollution due to fossil based energy production, the unsecure future availability and cost of fossil resources, the effort to become independent on fossil and other non-renewable resources etc.
In addition there is e.g. a EU binding renewable energy target of at least 32\% reduction of greenhouse gases by 2030 and for becoming a climate neutral continent by 2050  established by the European Commission \cite{EC2022}.

Beside the hydroelectricity as ``classical'' renewable energy resource over the last decades there was a significant increase in the use of other renewable resources, such as solar, wind, geothermal energy etc. In the solar sector the two pilars are photovoltaics and solar thermal techniques. 
The former transforms radiative energy directly into electrical energy and typically needs to be combined with batteries as storage medium for the electrical energy. The latter transforms radiative energy into thermal energy (heat) which can be stored and used on demand. At the level of big power plants a standard setup are (thermal) concentration solar power (CSP) plants. The mostly developed and used CSP technique are parabolic trough power plants (PPTP). There the heat can be stored with high temperatures over hours and days. 
The thermal energy is used to produce electrical energy on demand. Alternatively, it can be used as industrial heat avoiding big efficiency losses due to transformation bewteen different types of energy. 

A typical PPTP consists of a large network of collector pipes located in the focus of parabolic mirrors. The solar radiation is concentrated on the collector pipes where a heat transfer fluid (HTF) is heated
while being pumped through the pipes. The HTF supposed to reach very high temperatures (up to $400^\circ\text{C}$) is used  to  heat a storage medium (typically salt) outside the collector field in a hugh isolated storage tank. From there thermal energy can be extracted on demand e.g. by powering a steam turbine for producing electrical energy. The storage times are typically hours and days to compensate the daily displacement between supply and demand and in addition to compensate potential reduced supply or increased demand beyond the daily cycle.  

The first commercial power plant of this type was built in 1984 (SEGS I) in the US \cite{SEGS}. 
Meanwhile this plant was enlarged up to SEGS IX and has become one of the biggest PPTP's on the world, besides the Solana Generating Station in the US \cite{SGS}. 
The United States and Spain are the leading countries in terms of annual energy production via PPTP's. 
An important African project worth to be mentioned are the NOOR I and II PPTP's in Quarzazate, Marocco \cite{Aqachmar2019,NOORI+II}. 
However, there is many new big PTPP's in planning or under construction in many countries with big annual solar radiative input \cite{PtPPlist}.     
Typically, the collector fields of such big PPTP's cover a few hundreds of hectare of land, they have a power production of a few hundred Mega Watts 
and an annual energy production (e.g. NOOR I and II) reaching the range of a Tera Watt hour ($1 \, TWh/a$).

As far as HTF's is concerned nowadays special oils are used. In the past also water or salt were used due to their high heat capacity. The disadvantage for water are the high pressures which arise at such high temperatures. In case of salt the problem lies in high melting point at above  $220^\circ\text{C}$. 
There is ongoing reseach on possible alternative HTF's, such as nanofluids. We refer to a recent overview paper on this issue \cite{Nawsud2022}.

As far as research for PPTP's models is concerned there is variuos directions. A detailed description of various PPTP reseach directions can be found in the introduction of \cite{BGSP}. There is models for the HTF in single collector pipes based on the energy balance. These models describe the temperature dynamics in the pipe, but the density or pressure dynamics is not considered. On the other hand there is the aim for a complete PPTP network model. There is literature considering the complete PPTP network, mostly with strongly simplified models (e.g. space homogeneous in the single collector pipes). 
A recent general modeling, simulation and optimsation (MSO) approach is given in \cite{BGSP}. There from general thermo-fluid dynamic mass, momentum and energy balance laws an asymptotic (non space homogeneous) model was derived.  On one hand this model still includes the main physical effects and on the other hand it is simple enough for fast and robust simulations which are a crucial prerequisite for  optimisation strategies. In fact, in  \cite{BGSP}  optimisation results are shown for both the planning and the operational phase. Based in the model from  \cite{BGSP} alternative HTF fluids are considered in  \cite{BG23b}.
Let us mention that similar simplifying and asymptotic approaches where used 
for other thermo-fluid dynamics applications, such as gas the dynamics in gas pipelines \cite{BGH11} or in exhaust tubes \cite{GR13,GRW14}.

Here we consider the asymptotic thermo-fluid dynamic model derived and studied in \cite{BGSP}. To start with we consider the well posedness problem of the model for a single collector pipe. This is the basis for further studies on the network problem. To our knowledge there is no well posedness results for this system of nonlinear equations.

The paper is organized as follows. In Section \ref{sec:singlepipe} we present the mathematical problem and the main results. In Section \ref{sec:existence}
the proof for the existence result is given. Section \ref{stationaryproblem} refers to the { long-time behaviour and the} stationary problem.
In Appendix \ref{assumption} the validity of the taken assumptions from the viewpoint of the real world PPTP application are discussed. 
We end the paper with a Conclusion.
%
%
\section{Thermo-fluid dynamics in a single pipe}
\label{sec:singlepipe}
We consider a HTF liquid in a single (thin and long) collector pipe. In \cite{BGSP} a one-dimensional asymptotic thermo-fluid dynamic model was derived. The model describes quantities -- intended as averages over the cross section of the pipe -- which depend on time $t$ and the one dimensional longitudinal extension $x$  of the collector pipe. 
We denote by $\rho,\,u,\,p,\,T$ the density, the velocity, the pressure, the temperature of the HTF fluid, respectively. Then the  model reads
\begin{equation}
\label{eq:PP}
\begin{cases}
\pt\rho+\px(\rho u)=0,&\quad t>0,\,x\in(0,1),\\
\px p=-\alpha \rho u|u|,&\quad t>0,\,x\in(0,1),\\
\px u={\rho^{-2}}(q(t,x)-\beta_1(T-T_{out})-\beta_2(T^4- T_{sky}^4)),&\quad t>0,\,x\in(0,1),\\
T=\gamma-\rho,&\quad t>0,\,x\in(0,1),
\end{cases}
\end{equation}
where $\alpha,\,\beta_1,\,\beta_2,\,\gamma>0$ are constants. The quantities $q, T_{out}, T_{sky}$ are given functions denoting the solar radiative source,
the outside temperature and the background sky temperature, respectively.
The equations (\ref{eq:PP}) represent mass conservation, the momentum  
and the energy balance. The right hand side of the  (second) momentum balance equation is the surface friction term. The right hand side of the third equation - the asymptotic remainder of  the energy balance - contains
the solar radiative source term $q$, the convective loss term $-\beta_1(T-T_{out})$ and the radiative loss term $-\beta_2(T^4- T_{sky}^4)$. 

To shorten the notation we introduce 
\begin{equation}
	f(t,x)  =  q(t,x) + \beta_1 T_{out} + \beta_2 T_{sky}^4, \\
	\label{f01} 
\end{equation}
such that 
\begin{equation}
	q(t,x)-\beta_1(T-T_{out})-\beta_2(T^4- T_{sky}^4)
	=  
	f(t,x) - \beta_1 T - \beta_2 T^4. 
	\label{f02} 
\end{equation}

We augment \eqref{eq:PP} with the initial data
\begin{equation}
\label{eq:CL-ic}
\begin{cases}
u(0,x)=u_0(x),&\quad x\in(0,1),\\
p(0,x)=p_0(x),&\quad x\in(0,1),\\
T(0,x)=T_0(x),&\quad x\in(0,1),
\end{cases}
\end{equation}
and the boundary conditions
\begin{equation}
\label{eq:CL-bc}
\begin{cases}
\rho(t,0)=\rho_l(t),&\quad t>0,\\
\rho(t,1)=\rho_r(t),&\quad t>0,\\
p(t,0)=p_l(t),&\quad t>0,\\
p(t,1)=p_r(t),&\quad t>0.
\end{cases}
\end{equation}
Using the fact that the inverse of the function 
$$G(\xi)=\xi\abs{\xi}$$
is the odd function 
$$g(\xi)=\sgn{\xi}\sqrt{\abs{\xi}}$$
we can reformulate the 
initial-boundary value problem \eqref{eq:PP}, \eqref{eq:CL-ic}, \eqref{eq:CL-bc} as an elliptic-hyperbolic system in $\rho$ and $p$ as follows
\begin{equation}
\label{eq:CL}
\begin{cases}
\pt\rho-\px\left(\frac{g\left(\px p\right)}{\sqrt{\alpha}}\sqrt{{\rho}}\right)=0,&\quad t>0,\,x\in(0,1),\\
-\px \left(g\left(\frac{\px p}{\alpha\rho}\right)\right)=F(t,x,\rho),&\quad t>0,\,x\in(0,1),\\
\rho(t,0)=\rho_l(t),&\quad t>0,\\
\rho(t,1)=\rho_r(t),&\quad t>0,\\
p(t,0)=p_l(t),&\quad t>0,\\
p(t,1)=p_r(t),&\quad t>0,\\
p(0,x)=p_0(x),&\quad x\in(0,1),\\
\rho(0,x)=\rho_0(x),&\quad x\in(0,1),
\end{cases}
\end{equation}
where 
\begin{equation*}
F(t,x,\rho)=\frac{f(t,x)-\beta_1(\gamma-\rho)-\beta_2(\gamma-\rho)^4}{\rho^2},\qquad \rho_0(x)=\gamma-T_0(x).
\end{equation*}

We assume
\begin{align}
\label{eq:ass1} 
&f\in C^2([0,\infty)\times[0,1]),\quad\rho_0\in W^{1,\infty}(0,1),\quad \rho_l,\,\rho_r,\, p_l,\,p_r\in W^{1,\infty}(0,\infty);\\
\label{eq:ass3} 
&\exists\, \gamma_* >0 \>\>\text{s.t.}\>\>  0\le f\le\beta_1(\gamma-\gamma_*)+\beta_2(\gamma-\gamma_*)^4,\quad
 \gamma_*\le \rho_0,\, \rho_l,\,\rho_r\le \gamma. 
\end{align}
For a justification  of the last assumption (\ref{eq:ass3}) see Appendix \ref{assumption}. 
%

Inspired by \cite{zbMATH03651023} we use the following definition of solution.

\begin{definition}
\label{def:sol}
Let $\rho,\,p:[0,\infty)\times(0,1)\to\R$ be functions.
We say that the couple $(\rho,\,p)$ is an entropy solution of \eqref{eq:CL} if 
\begin{Definitions}
\item \label{def:rreg} $\rho\in L^\infty(0,T;BV(0,1)),$ $T>0$;
\item \label{def:rnonsing} $\gamma_*\le\rho\le\gamma$;
\item \label{def:preg} $p\in L^\infty(0,\infty;W^{1,\infty}(0,1)), \,\pxx p \in \mathcal{M}((0,T)\times(0,1)),\,T>0$;
\item \label{def:elleq} for every test function $\vfi\in C^\infty(\R\times(0,1))$ with compact support
\begin{equation}
\label{eq:def-ell}
\int_0^\infty\!\!\int_0^1 g\left(\frac{\px p}{\alpha\rho}\right)\px\vfi dtdx=\int_0^\infty\!\!\int_0^1 F(t,x,\rho)\vfi dtdx;
\end{equation}
\item \label{def:ell-bc} $p(t,0)=p_l(t)$ and $p(t,1)=p_r(t)$ for almost every $t\ge0$;
\item \label{def:hypeq} for every nonnegative test function $\vfi\in C^\infty(\R^2)$ with compact support and every constant $c\in\R$
\begin{equation}
\label{eq:def-hyp}
\begin{split}
\int_0^\infty\!\!\int_0^1 &\left(\abs{\rho-c}\pt\vfi+u\abs{\rho-c}\px\vfi-c\sgn{\rho-c}F(t,x,\rho(t,x))\right)dtdx\\
&-\int_0^\infty u(t,1)\sgn{\rho_r(t)-c}\left(\rho(t,1)-c\right)\vfi(t,1)dt\\
&+\int_0^\infty u(t,0)\sgn{\rho_l(t)-c}\left(\rho(t,0)-c\right)\vfi(t,0)dt\\
&+\int_0^1\abs{\rho_0(x)-c}\vfi(0,x)dx\ge0,
\end{split}
\end{equation}
where
\begin{equation}
\label{eq:defu}
u=-g\left(\frac{\px p}{\alpha\rho}\right).
\end{equation}
\end{Definitions}
\end{definition}

Arguing in \cite{zbMATH03651023} we can prove that
\begin{align*}
u(t,0)\ge 0&\Rightarrow  \rho(t,0)=\rho_l(t),\\
u(t,1)\le 0&\Rightarrow \rho(t,1)=\rho_r(t).
\end{align*}

The main results of this paper are the following existence result for \eqref{eq:CL}.
%
%
\begin{theorem}
\label{th:main}
Assume \eqref{eq:ass1} and \eqref{eq:ass3}. 
The initial boundary value problem \eqref{eq:CL} admits an entropy solution $(\rho,\,p)$
in the sense of Definition \ref{def:sol}.
\end{theorem}
%
%
The formal limit problem of (\ref{eq:PP}) as $t \rightarrow \infty$ is given by 
(the index $\infty$ denotes the variables of the limiting problem) 
\begin{equation}
	\label{eq:PP2lt}
	\begin{cases}
		\px(\rho_{\infty} u_{\infty})=0,&\quad \,x\in(0,1),\\
		\px p_{\infty}=-\alpha \rho_{\infty} u_{\infty}|u_{\infty}|,&\quad \,x\in(0,1),\\
		\px u_{\infty}={\rho_{\infty}^{-2}}(q_{\infty}(x)-\beta_1(T_{\infty}-T_{out})-\beta_2(T_{\infty}^4- T_{sky}^4)),&\quad \,x\in(0,1),\\
		T_{\infty}=\gamma-\rho_{\infty},&\quad \,x\in(0,1),
	\end{cases}
\end{equation}
where $q_{\infty} (x)= \lim\limits_{t\to \infty} q(t,x)$. We have boundary conditions
\begin{equation}
	\label{eq:PP2ltbc}
	\begin{cases}
		\rho_{\infty}(0)=\rho_{l,\infty}  = \lim\limits_{t\to \infty} \rho_l(t), &\mbox{ if } u(0)>0,\\
		\rho_{\infty}(1)=\rho_{r,\infty}  = \lim\limits_{t\to \infty} \rho_r(t), &\mbox{ if } u(1)<0,\\
		p_{\infty}(0)=p_{l,\infty}  = \lim\limits_{t\to \infty} p_l(t),&{}\\
		p_{\infty}(1)=p_{r,\infty}  = \lim\limits_{t\to \infty} p_r(t),&{}
	\end{cases}
\end{equation}
were we assume that the mentioned limits $q_{\infty}, \rho_{l,\infty}, \rho_{r,\infty}, p_{l,\infty}, p_{r,\infty}$ exist. In addition 
$f_{\infty}(x) = q_{\infty}(x) + \beta_1 T_{out} + \beta_2 T_{sky}^4$.
%
%
According to \eqref{eq:CL} we can reformulate the limit problem as an elliptic-hyperbolic system in $\rho_\infty$ and $p_\infty$ as follows
\begin{equation}
	\label{eq:CLlongtime}
	\begin{cases}
		\px\left(\frac{g\left(\px p_\infty\right)}{\sqrt{\alpha}}\sqrt{{\rho_\infty}}\right)=0,&\quad x\in(0,1),\\
		-\px \left(g\left(\frac{\px p_\infty}{\alpha\rho_\infty}\right)\right)=F_\infty(x,\rho_\infty),
		&\quad x\in(0,1),\\
		\rho_\infty(0)=\rho_{l,\infty},&\\
		\rho_\infty(1)=\rho_{r,\infty},&\\
		p_\infty(0)=p_{l,\infty},&\\
		p_\infty(1)=p_{r,\infty}.&
	\end{cases}
\end{equation}
where 
\begin{equation*}
	F_\infty(x,\rho)=\frac{f_\infty(x)-\beta_1(\gamma-\rho_\infty)-\beta_2(\gamma-\rho_\infty)^4}{\rho_\infty^2}.
\end{equation*}
The we can formulate the following result on the long time behavior, where we shall use the following definition of solution for \eqref{eq:CLlongtime}.

\begin{definition}
\label{def:sol-long}
Let $\rho_\infty,\,p_\infty:(0,1)\to\R$ be functions.
We say that the couple $(\rho_\infty,\,p_\infty)$ is an entropy solution of \eqref{eq:CLlongtime} if 
\begin{DefinitionsE}
\item \label{def:rreg-long} $\rho_\infty\in BV(0,1)$;
\item \label{def:rnonsing-long} $\gamma_*\le\rho_\infty\le\gamma$;
\item \label{def:preg-long} $p_\infty\in W^{1,\infty}(0,1),\,\pxx p_\infty\in \mathcal{M}(0,1)$;
\item \label{def:elleq-long} for every test function $\vfi\in C^\infty((0,1))$ with compact support
\begin{align}
\label{eq:def-ell-long}
\int_0^1& \left(\frac{g\left(\px p_\infty\right)}{\sqrt{\alpha}}\sqrt{{\rho_\infty}}\right)\px\vfi dx=0,\\
\int_0^1& g\left(\frac{\px p_\infty}{\alpha\rho}\right)\px\vfi dx=\int_0^1 F_\infty(x,\rho_\infty)\vfi dx;
\end{align}
\item \label{def:ell-bc-long} $p_\infty(0)=p_{l,\infty}$ and $p_\infty(1)=p_{r,\infty}$ in the sense of traces.
\end{DefinitionsE}
\end{definition}

\begin{remark}
\label{rm:BV}
It is important to remember that
\begin{equation*}
\pxx p_\infty\in \mathcal{M}(0,1)\Leftrightarrow \px p_\infty\in BV(0,1).
\end{equation*}
{ and consequently 
\begin{equation*}
	\pxx p \in L^{\infty}(0,\infty;\mathcal{M}(0,1))\Leftrightarrow \px p_\infty\in L^{\infty}(0,\infty;BV(0,1)).
\end{equation*}
}
\end{remark}

%
%
\begin{theorem}
	\label{th:main02}
Assume \eqref{eq:ass1} and \eqref{eq:ass3}. Let $(\rho,\,p)$ be an entropy solution of \eqref{eq:CL}  in the sense of Definition \ref{def:sol}.
There exist $\{t_k\}_k\subset(0,\infty),\,t_k\to\infty,\,  \rho_{\infty} \in L^\infty(0,1),\,u_\infty,\,p_\infty\in W^{1,\infty}(0,1)$ such that
	\begin{align}
		\rho(t_k,\cdot)& \rightharpoonup   \rho_{\infty} 
		\text{ weakly in } L^r(0,1), \, 1 < r < \infty,  \label{stab01} \\
		u(t_k,\cdot) & \rightarrow  u_{\infty} 
		\text{ strongly in } C([0,1])
		\text{and  weakly in } W^{1,r}(0,1), \, 1 < r < \infty,   \label{stab02} \\
		p(t_k,\cdot) & \rightharpoonup  p_{\infty} \text{ strongly in } C([0,1])
		\text{and  weakly in } W^{1,r}(0,1), \, 1 < r < \infty.  \label{stab04} 
	\end{align}    
	If  in addition  we assume that $\pxx p$  belong in $L^{\infty}(0,\infty;\mathcal{M}(0,1))$ we can conclude that the convergence is along all the family $t\to\infty$,
\begin{align}
	\rho(t,\cdot)& \to   \rho_{\infty} 
		\text{ strongly in } L^r(0,1), \, 1 < r < \infty,  \label{eq:stab01} \\
		u(t,\cdot) & \rightarrow  u_{\infty} 
		\text{ strongly in } W^{1,r}(0,1), \, 1 < r < \infty,   \label{eq:stab02} \\
\label{eq:stab05}
p(t,\cdot) &\to  p_{\infty} 
		\text{ strongly in } W^{1,r}(0,1), \, 1 < r < \infty,  
\end{align}	
and 	the limit  couple $(\rho_\infty,\,p_\infty)$ is the {unique}  entropy solution of \eqref{eq:CLlongtime} 
	in the sense of Definition \ref{def:sol-long}. 
%
\end{theorem}
\begin{remark}
\label{remark:main}
From Theorem \ref{th:main02} we see that the result on the  long time behaviour is not complete. However, assuming a slightly better bound on the pressure $p$  we can prove the convergence to (\ref{eq:CLlongtime}).
\end{remark}
%
%
\section{Proof of Theorem \ref{th:main}}
\label{sec:existence}
The main result of this section is the  existence result for \eqref{eq:CL} stated in Theorem \ref{th:main}.
Our argument is based on a compactness result on the family of solutions $\{(\re,\,\pe)\}_{\eps>0}$ of the following vanishing viscosity type approximation of 
\eqref{eq:CL}

\begin{equation}
\label{eq:CLeps}
\begin{cases}
\pt\re-\px\left(\frac{g\left(\px \pe\right)}{\sqrt{\alpha}}\sqrt{{\re}}\right)=\eps\pxx\re,&\quad t>0,\,x\in(0,1),\\
-\px \left(g\left(\frac{\px \pe}{\alpha\re}\right)\right)=F(t,x,\re),&\quad t>0,\,x\in(0,1),\\
\re(t,0)=\rho_{l,\eps}(t),&\quad t>0,\\
\re(t,1)=\rho_{r,\eps}(t),&\quad t>0,\\
\pe(t,0)=p_{l,\eps}(t),&\quad t>0,\\
\pe(t,1)=p_{r,\eps}(t),&\quad t>0,\\
p(0,x)=p_{0,\eps}(x),&\quad x\in(0,1),\\
\rho(0,x)=\rho_{0,\eps}(x),&\quad x\in(0,1),
\end{cases}
\end{equation}
where $\rho_{l,\eps},\, \rho_{r,\eps},\,p_{l,\eps},\,p_{r,\eps},\, p_{0,\eps},\, \rho_{0,\eps}$ are smooth approximations of 
$\rho_{l},\, \rho_{r},\,p_{l},\,p_{r},\, p_{0},\, \rho_{0}$ satisfying \eqref{eq:ass1} and \eqref{eq:ass3}.
The existence of smooth solutions for \eqref{eq:CLeps} can be proved following the same argument of 
\cite{zbMATH07144663, zbMATH07634966, zbMATH02215273}.

Introducing the function
\begin{equation}
\label{eq:defueps}
\ue=-g\left(\frac{\px \pe}{\alpha\re}\right)
\end{equation}
the following equations hold
\begin{equation}
\label{eq:PPeps}
\begin{cases}
\pt\re+\px(\re \ue)=\eps\pxx\re,&\quad t>0,\,x\in(0,1),\\
\px \pe=-\alpha \re \ue|\ue|,&\quad t>0,\,x\in(0,1),\\
\px \ue=F(t,x,\re),&\quad t>0,\,x\in(0,1).
\end{cases}
\end{equation}

In what follows we denote with $c$ and $C$ all the constants independent on $t$ and $\eps$.

The following lemmas are needed.

\begin{lemma}
\label{lm:maxp}
The following estimates hold
\begin{equation}
\label{eq:maxp}
0<\gamma_*\le\re\le\gamma.
\end{equation}
\end{lemma}

\begin{proof}
We can rewrite the first equation in \eqref{eq:CLeps} as follows
\begin{equation}
\label{eq:r}
\pt\re-g\left(\frac{\px \pe}{\alpha\re}\right)\px\re+\re F(t,x,\re)=\eps\pxx\re.
\end{equation}
Therefore $\re$ is the solution of the following initial boundary value problem
\begin{equation}
\label{eq:v}
\begin{cases}
\displaystyle\pt v-g\left(\frac{\px \pe}{\alpha\re}\right)\px v+v F(t,x,v)=\eps\pxx v,&\quad t>0,\,x\in(0,1),\\[5pt]
v(t,0)=\rho_{l,\eps}(t),&\quad t>0,\\
v(t,1)=\rho_{r,\eps}(t),&\quad t>0,\\
v(0,x)=\rho_{0,\eps}(x),&\quad x\in(0,1).
\end{cases}
\end{equation}
Since, thanks to \eqref{eq:ass3}, $\gamma$ and $\gamma_*$ provide a super and a subsolution for \eqref{eq:v}, the claim is proved.
\end{proof}

\begin{lemma}
\label{lm:uBV} 
The family $\{\px\re\}_{\eps>0}$ is bounded in $L^{\infty}(0,T;L^1(0,1))$ for every $T\ge0$.
\end{lemma}

\begin{proof}
Using \eqref{eq:ass1}, \eqref{eq:ass3}, \eqref{eq:r} and Lemma \ref{lm:maxp}
\begin{align*}
\frac{d}{dt}\int_0^1&\abs{\px\re}dx=\int_0^1\sgn{\px\re}\ptx\re dx\\
=&\sgn{\px\re(t,1)}\pt\rho_{r,\eps}(t)-\sgn{\px\re(t,0)}\pt\rho_{l,\eps}(t)-\int_0^1\pt\re\pxx\re\delta_{\{\px\re=0\}}dx\\
=&{\sgn{\px\re(t,1)}\pt\rho_{r,\eps}(t)}
{-\sgn{\px\re(t,0)}\pt\rho_{l,\eps}(t)} \underbrace{-\eps\int_0^1(\pxx\re)^2\delta_{\{\px\re=0\}}dx}_{\le0}\\
&\underbrace{-\int_0^1 g\left(\frac{\px \pe}{\alpha\re}\right)\px\re\pxx\re\delta_{\{\px\re=0\}}dx}_{=0}+
\int_0^1\re F(t,x,\re)\underbrace{\pxx\re\delta_{\{\px\re=0\}}}_{=\px\sgn{\px\re}}dx\\
\le &\abs{\pt\rho_{r,\eps}(t)}+\abs{\pt\rho_{l,\eps}(t)}+\rho_{r,\eps}(t) F(t,1,\rho_{r,\eps}(t))\sgn{\px\re(t,1)}\\
&-\rho_{l,\eps}(t) F(t,0,\rho_{l,\eps}(t))\sgn{\px\re(t,0)}\\
&+\int_0^1\abs{\px\re}\left( F(t,x,\re)+\re \partial_\rho F(t,x,\re)\right)dx\\
&+\int_0^1\re\px F(t,x,\re)\sgn{\px\re}dx\\
\le &\abs{\pt\rho_{r,\eps}(t)}+\abs{\pt\rho_{l,\eps}(t)}+\rho_{r,\eps}(t) \abs{F(t,1,\rho_{r,\eps}(t))}+\rho_{l,\eps}(t) \abs{F(t,0,\rho_{l,\eps}(t))}\\
&+\int_0^1\abs{\px\re}\abs{F(t,x,\re)+\re \partial_\rho F(t,x,\re)}dx+\int_0^1\re\abs{\px F(t,x,\re)}dx\\
\le &c+c\int_0^1\abs{\px\re}dx,
\end{align*}
where $\delta_{\{\px\re=0\}}$ is the Dirac delta concentrated on the set ${\{\px\re=0\}}$.
The Gronwall Lemma gives the claim.
\end{proof}

\begin{lemma}
\label{lm:pxl32}
The family $\{\pe\}_{\eps>0}$ is bounded in $L^\infty(0,\infty;W^{1,3/2}(0,1))$.
\end{lemma}

\begin{proof}
Define
\begin{equation*}
\wpe(t,x)=(1-x)p_{l,\eps}(t)+xp_{r,\eps}(t).
\end{equation*}
Multiply the second equation in \eqref{eq:CLeps} by $\pe-\wpe$ and integrate over $(0,1)$
\begin{equation*}
-\int_0^1\px \left(g\left(\frac{\px \pe}{\alpha\re}\right)\right)(\pe-\wpe)dx=\int_0^1F(t,x,\re)(\pe-\wpe)dx.
\end{equation*}
Since
\begin{equation*}
(\pe-\wpe)(t,0)=(\pe-\wpe)(t,1)=0,
\end{equation*}
we have 
\begin{equation}
\label{eq:pxl32.1}
\int_0^1 g\left(\frac{\px \pe}{\alpha\re}\right)\px\pe dx
=\int_0^1 g\left(\frac{\px \pe}{\alpha\re}\right)\px\wpe dx
+\int_0^1F(t,x,\re)(\pe-\wpe)dx.
\end{equation}
Due to the H\"older Inequality, the definition of  $g$ and Lemma \ref{lm:maxp}
\begin{equation}
\label{eq:pxl32.2}
\int_0^1 g\left(\frac{\px \pe}{\alpha\re}\right)\px\pe dx\ge\frac{1}{\alpha\gamma}\int_0^1|\px\pe|^{3/2}dx.
\end{equation}
Thanks to the definition of the function $g$, \eqref{eq:ass1} and Lemma \ref{lm:maxp} 
\begin{equation}
\label{eq:pxl32.3}
\begin{split}
\int_0^1& g\left(\frac{\px \pe}{\alpha\re}\right)\px\wpe dx+\int_0^1F(t,x,\re)(\pe-\wpe)dx\\
\le&\frac{\norm{\px\wpe}_{L^\infty((0,\infty)\times(0, 1))}}{\sqrt{\alpha\gamma_*}}\int_0^1\sqrt{\abs{\px\pe}}dx\\
&+\norm{F}_{L^\infty((0,\infty)\times(0, 1)\times(\gamma_*,\gamma))}\int_0^1|\pe|dx\\
&+\norm{F}_{L^\infty((0,\infty)\times(0, 1)\times(\gamma_*,\gamma))}\norm{\wpe}_{L^\infty((0,\infty)\times(0, 1))}\\
\le&c\left(\int_0^1{\abs{\px\pe}}^{3/2}dx\right)^{1/3}+c\int_0^1|\px\pe|dx+c\\
\le&c\left(\int_0^1{\abs{\px\pe}}^{3/2}dx\right)^{1/3}+c\left(\int_0^1{\abs{\px\pe}}^{3/2}dx\right)^{2/3}+c.
\end{split}
\end{equation}
Using \eqref{eq:pxl32.2} and \eqref{eq:pxl32.3} in \eqref{eq:pxl32.1} we get
\begin{equation*}
\int_0^1|\px\pe|^{3/2}dx\le c\left(\int_0^1{\abs{\px\pe}}^{3/2}dx\right)^{1/3}+c\left(\int_0^1{\abs{\px\pe}}^{3/2}dx\right)^{2/3}+c,
\end{equation*}
therefore  $\{\px\pe\}_{\eps>0}$ is bounded in $L^\infty(0,\infty;L^{3/2}(0,1))$. Thanks to \eqref{eq:ass1}, we get the claim.
\end{proof}

\begin{lemma}
\label{lm:ul2} 
The family $\{\sqrt{\eps}\px\re\}_{\eps>0}$ is bounded in $L^{2}((0,T)\times(0,1))$ for every $T\ge0$.
\end{lemma}

\begin{proof}
Define
\begin{equation*}
\wre=(1-x)\rho_{l,\eps}(t)+x\rho_{r,\eps}(t)
\end{equation*}
Multiply the second equation in \eqref{eq:r} by $\re-\wre$ and integrate over $(0,T)\times(0,1)$
\begin{align*}
\int_0^T\int_0^1&\pt\re(\re-\wre)dtdx-\int_0^T\int_0^1g\left(\frac{\px \pe}{\alpha\re}\right)\px\re(\re-\wre)dtdx\\
&+\int_0^T\int_0^1\re F(t,x,\re)(\re-\wre)dtdx=\eps\int_0^T\int_0^1\pxx\re(\re-\wre)dtdx.
\end{align*}
Integrating by parts and rearranging the terms we get
\begin{align*}
\eps\int_0^T&\int_0^1(\px\re)^2dtdx=-\int_0^1\frac{(\re-\wre)^2(T,x)}{2}dx+\int_0^1\frac{(\re-\wre)^2(0,x)}{2}dx\\
&-\int_0^T\int_0^1 \pt\wre(\re-\wre)dtdx+\int_0^T\int_0^1 F(t,x,\re)\frac{(\re-\wre)^2}{2}dtdx\\
&+\int_0^T\int_0^1g\left(\frac{\px \pe}{\alpha\re}\right)\px\wre(\re-\wre)dtdx
-\int_0^T\int_0^1\re F(t,x,\re)(\re-\wre)dtdx\\
&+\eps\int_0^T\int_0^1\px\re\px\wre dtdx.
\end{align*}
Using \eqref{eq:ass1}  Lemmas \ref{lm:maxp}, \ref{lm:pxl32} we get
\begin{align*}
\eps\int_0^T&\int_0^1(\px\re)^2dtdx\le c+cT+c\int_0^T\int_0^1\sqrt{\abs{\px\pe}}dtdx+\frac{\eps}{2}\int_0^T\int_0^1(\px\re)^2dtdx\\
&\le c+cT
+cT\norm{\px\pe}_{L^\infty(0,\infty;L^{3/2}(0,1))}
+\frac{\eps}{2}\int_0^T\int_0^1(\px\re)^2dtdx
\\
&\le c+cT+\frac{\eps}{2}\int_0^T\int_0^1(\px\re)^2dtdx,
\end{align*}
that gives the claim.
\end{proof}

\begin{lemma}
\label{lm:rcomp}
There exist a subsequence $\{\eps_k\}_k\subset(0,\infty),\,\eps_k\to0$ and a function $\rho$ such that
\begin{equation}
\label{eq:rcomp}
\begin{split}
&\rho\in L^\infty(0,T;BV(0,1)),\,T>0,\\
&\gamma_*\le\rho\le\gamma,\\
&\rho_{\eps_k}\to\rho,\quad\text{in $L^r_{loc}((0,\infty)\times(0,1)),\,1\le r<\infty,$ and a.e. in $(0,\infty)\times(0,1)$ as $k\to\infty$}.
\end{split}
\end{equation}
\end{lemma}

\begin{proof}
Given $T>0$.
Thanks to Lemmas \ref{lm:maxp}, \ref{lm:uBV}, and \ref{lm:ul2}, we can apply the Aubin-Lions Lemma \cite[Theorem 5]{Si} $\{\re\}_{\eps>0}$.
Indeed $\{\re\}_{\eps>0}$ is uniformly bounded in $L^\infty(0,T;BV(0,1))$ and, using the equation, 
$\{\pt\re\}_{\eps>0}$ is uniformly bounded in $L^2(0,T;H^{-1}(0,1))$. As a consequence there exists a function 
$\rho$  satisfying \eqref{eq:rcomp}.
\end{proof}

\begin{lemma}
\label{lm:ulip}
The family $\{\ue\}_{\eps>0}$ is bounded in $L^\infty(0,\infty;W^{1,\infty}(0,1))$.
\end{lemma}

\begin{proof}
Thanks to \eqref{eq:ass1}, \eqref{eq:PPeps} and Lemma \ref{lm:maxp}
\begin{equation*}
\px\ue \in L^\infty((0,\infty)\times(0,1)).
\end{equation*}
We have to prove that
\begin{equation}
\label{eq:ulip1}
\ue \in L^\infty((0,\infty)\times(0,1)).
\end{equation}
Define
\begin{equation}
\label{eq:mK}
m_\eps(t)=\ue(t,0),\qquad K_\eps(t,x)=\int_0^x F(t,y,\re(t,y))dy.
\end{equation}
We have 
\begin{equation}
\label{eq:ulip2}
\ue(t,x)=m_\eps(t)+ K_\eps(t,x).
\end{equation}
Since, thanks to \eqref{eq:ass1}, \eqref{eq:ass3}, \eqref{eq:PPeps} and Lemma \ref{lm:maxp}, 
\begin{equation}
\label{eq:ulip3}
K_\eps\in L^\infty((0,\infty)\times(0,1)),
\end{equation}
in order to prove \eqref{eq:ulip1} we have to prove 
\begin{equation}
\label{eq:ulip4}
\me \in L^\infty(0,\infty).
\end{equation}
Integrating the second equation in \eqref{eq:PPeps} in $(0,1)$,
 using \eqref{eq:maxp}, and the monotonicity of the map $(\xi,\zeta)\mapsto(\xi+\zeta)\abs{\xi+\zeta}$
in both the variables we have
\begin{align*}
p_l(t)-p_r(t)=&\alpha\int_0^1\re(t,y)(m_\eps(t)+ K_\eps(t,x))\abs{m_\eps(t)+ K_\eps(t,x)}dx\\
\le&\alpha\gamma\left(m_\eps(t)+ \norm{K_\eps}_{L^\infty((0,\infty)\times(0,1))}\right)\abs{m_\eps(t)+ \norm{K_\eps}_{L^\infty((0,\infty)\times(0,1))}},\\
p_l(t)-p_r(t)=&\alpha\int_0^1\re(t,y)(m_\eps(t)+ K_\eps(t,x))\abs{m_\eps(t)+ K_\eps(t,x)}dx\\
\ge&\alpha\gamma_*\left(m_\eps(t)- \norm{K_\eps}_{L^\infty((0,\infty)\times(0,1))}\right)\abs{m_\eps(t)- \norm{K_\eps}_{L^\infty((0,\infty)\times(0,1))}}.
\end{align*}
Applying the function $g$ to both sides we get
\begin{align*}
g\left(\frac{p_l(t)-p_r(t)}{\alpha\gamma}\right)&-\norm{K_\eps}_{L^\infty((0,\infty)\times(0,1))}\\
\le&m_\eps(t)\le g\left(\frac{p_l(t)-p_r(t)}{\alpha\gamma_*}\right)+\norm{K_\eps}_{L^\infty((0,\infty)\times(0,1))}.
\end{align*}
Thanks to \eqref{eq:ulip3}, we have \eqref{eq:ulip4} and then the claim.
\end{proof}

\begin{lemma}
\label{lm:plip}
The family $\{\pe\}_{\eps>0}$ is bounded in $L^\infty(0,\infty;W^{1,\infty}(0,1))$.
\end{lemma}

\begin{proof}
Thanks to \eqref{eq:ass1}, \eqref{eq:PPeps} and Lemmas \ref{lm:maxp}, \ref{lm:ulip}
\begin{equation}
\label{eq:ulip1*}
\px\pe \in L^\infty((0,\infty)\times(0,1)).
\end{equation}
We have to prove that
\begin{equation}
\label{eq:ulip2}
\pe \in L^\infty((0,\infty)\times(0,1)).
\end{equation}
Since
\begin{equation*}
\pe(t,x)=p_l(t)+\int_0^x\px\pe(t,y)dy,
\end{equation*}
\eqref{eq:ulip2} follows from \eqref{eq:ass1} and \eqref{eq:ulip1*}.
\end{proof}

\begin{lemma}
\label{lm:pxxl1}
The family $\{\pxx\pe\}_{\eps>0}$ is bounded in $L^1((0,T)\times(0,1))$ for every $T\ge0$.
\end{lemma}

\begin{proof}
A direct computation gives
\begin{equation*}
\pxx\pe=-\alpha \ue\abs{\ue}\px\re-\alpha \re\abs{\ue}\px\ue.
\end{equation*}
Therefore, the claim follows from Lemmas \ref{lm:maxp}, \ref{lm:uBV}, and \ref{lm:plip}.
\end{proof}

%

\begin{lemma}
\label{lm:pcomp}
Let $\{\eps_k\}_k$ and $\rho$ be the ones introduced in Lemma \ref{lm:rcomp}.
There exists a function $p\in L^\infty(0,\infty;W^{1,\infty}(0,1)),\,\pxx p\in \mathcal{M}((0,T)\times(0,1)),\,T>0$ such that
\begin{equation}
\label{eq:pcomp}
\begin{split}
&p_{\eps_k}\to p,\quad\text{in $L^r_{loc}((0,\infty)\times(0,1)),\,1\le r<\infty,$ and a.e. in $(0,\infty)\times(0,1)$ as $k\to\infty$},\\
&p_{\eps_k}(t,\cdot)\to p(t,\cdot),\quad\text{unifrmly in $(0,1)$ for a.e.  $t\ge0$ as $k\to\infty$},\\
&\px p_{\eps_k}\to \px p,\quad\text{in $L^r_{loc}((0,\infty)\times(0,1)),\,1\le r<\infty,$ and a.e. in $(0,\infty)\times(0,1)$ as $k\to\infty$}.
\end{split}
\end{equation}
\end{lemma}

\begin{proof}
Since the functions
\begin{align*}
\mathcal{G}_\eps(a)=&\alpha\int_0^1\re(t,y)G\left(a-\int_0^y F(t,z,\re(t,z))dz\right)dy,\\
\mathcal{G}(a)=&\alpha\int_0^1\rho(t,y)G\left(a-\int_0^y F(t,z,\rho(t,z))dz\right)dy
\end{align*}
are strictly increasing and diverging at infinity, there exists $a_\eps,\,a\in L^\infty(0,\infty)$ such that
\begin{equation}
\label{eq:a1}
\begin{split}
p_r(t)=&p_l(t)+\alpha\int_0^1\re(t,y)G\left(a_\eps(t)-\int_0^y F(t,z,\re(t,z))dz\right)dy,\\
p_r(t)=&p_l(t)+\alpha\int_0^1\rho(t,y)G\left(a(t)-\int_0^y F(t,z,\rho(t,z))dz\right)dy.
\end{split}
\end{equation}
We claim that
\begin{equation}
\label{eq:a2}
\begin{split}
p_\eps(t,x)=&p_l(t)+\alpha\int_0^x\re(t,y)G\left(a_\eps(t)-\int_0^y F(t,z,\re(t,z))dz\right)dy,\\
p(t,x)=&p_l(t)+\alpha\int_0^x\rho(t,y)G\left(a(t)-\int_0^y F(t,z,\rho(t,z))dz\right)dy,
\end{split}
\end{equation}
solve
\begin{equation}
\label{eq:a3}
\begin{cases}
\displaystyle-\px \left(g\left(\frac{\px \pe}{\alpha\re}\right)\right)=F(t,x,\re),&\quad t>0,\,x\in(0,1),\\
\displaystyle-\px \left(g\left(\frac{\px p}{\alpha\re}\right)\right)=F(t,x,\rho),&\quad t>0,\,x\in(0,1),\\
\pe(t,0)=p(t,0)=p_{l,\eps}(t),&\quad t>0,\\
\pe(t,1)=p(t,1)=p_{r,\eps}(t),&\quad t>0.
\end{cases}
\end{equation}
Thanks to \eqref{eq:a1} and \eqref{eq:a2}, the boundary conditions in \eqref{eq:a3} are trivially satisfied.
Differentiating with respect to $x$ both sides of \eqref{eq:a2} we get
\begin{equation}
\label{eq:a4}
\begin{split}
\px p_\eps(t,x)=&\alpha\re(t,x)G\left(a_\eps(t)-\int_0^x F(t,z,\re(t,z))dz\right),\\
\px p(t,x)=&\alpha\rho(t,x)G\left(a(t)-\int_0^x F(t,z,\rho(t,z))dz\right),
\end{split}
\end{equation}
Since $g$ is the inverse of $f$ we get
\begin{equation}
\label{eq:a4}
\begin{split}
g\left(\frac{\px p_\eps(t,x)}{\alpha\re(t,x)}\right)=&a_\eps(t)-\int_0^x F(t,z,\re(t,z))dz,\\
g\left(\frac{\px p(t,x)}{\alpha\rho(t,x)}\right)=&a(t)-\int_0^x F(t,z,\rho(t,z))dz.
\end{split}
\end{equation}
One final differentiation gives the equations in \eqref{eq:a3}.

Thanks to \eqref{eq:rcomp} and \eqref{eq:a2}, in order to prove \eqref{eq:pcomp} it suffices to show
\begin{equation}
\label{eq:acomp}
a_{\eps_k}\to a,\quad\text{in $L^r_{loc}(0,\infty),\,1\le r<\infty,$ and a.e. in $(0,\infty)$ as $k\to\infty$}.
\end{equation}

Thanks to \eqref{eq:rcomp} for almost every $t$
\begin{equation}
\label{eq:a5}
\rho_{\eps_k}(t,\cdot)\to \rho(t,\cdot),\quad\text{in $L^r(0,1),\,1\le r<\infty,$ and a.e. in $(0,1)$ as $k\to\infty$}.
\end{equation}
Let $\tau$ be such that \eqref{eq:a5} holds.
Due to \eqref{eq:ass1}, \eqref{eq:a2} and the definition of $G$, $\{a_{\eps_k}\left(\tau\right)\}_k$ is bounded, 
then passing to a subsequence, there exists $\widetilde a\in\R$ such that
\begin{equation*}
\lim\limits_{k\to\infty}a_{\eps_k}\left(\tau\right)=\widetilde a.
\end{equation*}
We have to prove that
\begin{equation}
\label{eq:a6}
\widetilde a =a\left(\tau\right).
\end{equation}
Sending $k\to \infty$ in the first of \eqref{eq:a1} and sutracting with the second we have
\begin{align*}
0=&\alpha\int_0^1\rho(\tau,y)\left[
G\left(\widetilde a-\int_0^y F(\tau,z,\rho(t,y)(\tau,z))dz\right)dy-G\left(a(\tau)-\int_0^y F(\tau,z,\rho(\tau,z))dz\right)\right]dy\\
=&\alpha\left(\widetilde a -a\left(\tau\right)\right)\int_0^1\int_0^1\rho(\tau,y)
G'\left(\theta\widetilde a+(1-\theta)a(\tau)-\int_0^y F(\tau,z,\rho(\tau,z))dz\right)dyd\theta\\
=&2\alpha\left(\widetilde a -a\left(\tau\right)\right)\int_0^1\int_0^1\rho(\tau,y)
\abs{\theta\widetilde a+(1-\theta)a(\tau)-\int_0^y F(\tau,z,\rho(\tau,z))dz}dyd\theta,
\end{align*}
and then, or \eqref{eq:a6} holds or
\begin{equation*}
\theta\widetilde a+(1-\theta)a(\tau)-\int_0^y F(\tau,z,\rho(\tau,z))dz=0,\quad \text{for every $\theta\in[0,1]$ and a.e $y\in[0,1].$}
\end{equation*}
and differentiating with respect to $\theta$ we get again \eqref{eq:a6}.
\end{proof}

Direct consequence of \eqref{eq:defueps} and the previous lemma is the following result.

\begin{lemma}
\label{lm:ucomp}
Let $\{\eps_k\}_k$, $\rho$, and $p$ be the ones introduced in Lemmas \ref{lm:rcomp} and \ref{lm:pcomp}. We have that
\begin{equation}
\label{eq:ucomp}
\begin{split}
&u_{\eps_k}\to u,\quad\text{in $L^r_{loc}((0,\infty)\times(0,1)),\,1\le r<\infty,$ and a.e. in $(0,\infty)\times(0,1)$ as $k\to\infty$},\\
&u_{\eps_k}(t,\cdot)\to u(t,\cdot),\quad\text{uniformly in $(0,1)$ for a.e.  $t\ge0$ as $k\to\infty$}
\end{split}
\end{equation}
where $u$ is defined in \eqref{eq:defu}.
\end{lemma}

\begin{proof}[Proof of Theorem \ref{th:main}]
We begin by verifying that the functions $(\rho,\,p)$ introduced in Lemmas \ref{lm:rcomp} and \ref{lm:pcomp}
provide a solution to \eqref{eq:CL} in the sense of Definition \ref{def:sol}.
\ref{def:rreg}, \ref{def:rnonsing}, and \ref{def:preg} are already stated in  Lemmas \ref{lm:rcomp} and \ref{lm:pcomp}.
\ref{def:ell-bc} follows from \ref{lm:pcomp} and the identities  $\pek(\cdot,0)=p_l(\cdot)$ and $\pek(\cdot,1)=p_r(\cdot)$.
Since for every test function $\vfi\in C^\infty(\R\times(0,1))$ with compact support we have
\begin{equation*}
\int_0^\infty\!\!\int_0^1 g\left(\frac{\px \pek}{\alpha\rho}\right)\px\vfi dtdx=\int_0^\infty\!\!\int_0^1 F(t,x,\rek)\vfi dtdx,
\end{equation*}
\ref{def:elleq} follows from Lemmas \ref{lm:rcomp} and \ref{lm:pcomp}.

We have to prove \ref{def:hypeq}. Given $c\in\R$, using  \eqref{eq:PPeps}, a direct computation gives, passing to a smooth approximation of the 
absolute value,
\begin{equation}
\label{eq:entr}
\pt\abs{\rek-c}+\px\left(\uek\abs{\rek-c}\right)+c\sgn{\rek-c}F(t,x,\rek)\le \eps_k\pxx\abs{\rek-c}.
\end{equation}
Let $\vfi\in C^\infty(\R^2)$ be a  nonnegative test function. Multiplying \eqref{eq:entr} by $\vfi$ and integrating over $(0,\infty)\times(0,1)$ we get
\begin{equation}
\label{eq:def-hyp-eps}
\begin{split}
\int_0^\infty\!\!\int_0^1 &\left(\abs{\rek-c}\pt\vfi+\uek\abs{\rek-c}\px\vfi-c\sgn{\rek-c}F(t,x,\rek)\vfi\right)dtdx\\
&-\eps_k\int_0^\infty\!\!\int_0^1\px\rek\sgn{\rek-c}\px\vfi dtdx\\
&-\int_0^\infty\left(\uek(t,1)\abs{\rho_r(t)-c}\vfi(t,1)-\uek(t,0)\abs{\rho_l(t)-c}\vfi(t,0)\right)dt\\
&+\eps_k\int_0^\infty\left(\px\rek(t,1)\sgn{\rho_r(t)-c}\vfi(t,1)-\px\rek(t,0)\sgn{\rho_l(t)-c}\vfi(t,0)\right)dt\\
&+\int_0^1\abs{\rho_0(x)-c}\vfi(0,x)dx\ge0.
\end{split}
\end{equation}
Using Lemma \ref{lm:uBV},  \ref{lm:rcomp}, and \ref{lm:ucomp} as $k\to\infty$
\begin{equation}
\label{eq:def-hyp-eps.1}
\begin{split}
\int_0^\infty\!\!\int_0^1 &\left(\abs{\rho-c}\pt\vfi+u\abs{\rho-c}\px\vfi-c\sgn{\rho-c}F(t,x,\rho)\vfi\right)dtdx\\
&-\int_0^\infty\left(u(t,1)\abs{\rho_r(t)-c}\vfi(t,1)-u(t,0)\abs{\rho_l(t)-c}\vfi(t,0)\right)dt\\
&+\lim\limits_k\Big[\eps_k\int_0^\infty\left(\px\rek(t,1)\sgn{\rho_r(t)-c}\vfi(t,1)\right.\\
&\qquad\qquad\qquad \left.-\px\rek(t,0)\sgn{\rho_l(t)-c}\vfi(t,0)\right)dt\Big]\\
&+\int_0^1\abs{\rho_0(x)-c}\vfi(0,x)dx\ge0.
\end{split}
\end{equation}
We claim that
\begin{equation}
\label{eq:def-hyp-eps.2}
\begin{split}
\lim\limits_k\eps_k&\int_0^\infty\left(\px\rek(t,1)\vfi(t,1)-\px\rek(t,0)\vfi(t,0)\right)dt\\
=&\int_0^\infty\left(u(t,1)\left(\rho_r(t)-\rho(t,1)\right)\vfi(t,1)-u(t,0)\left(\rho_l(t)-\rho(t,0)\right)\vfi(t,0)\right)dt.
\end{split}
\end{equation}

For every $n\in\N\setminus\{0\}$ consider a function $\chi_n$ such that
\begin{equation}
\label{eq:chi}
\begin{split}
&\chi_n\in C^\infty(\R),\qquad \chi_n(0)=\chi_n(1)=1,\qquad 0\le \chi_n\le 1,\\
&\abs{\chi_n'}\le n,\qquad x\in\left[\frac{1}{n},1-\frac{1}{n}\right]\Rightarrow \chi_n(x)=0.
\end{split}
\end{equation}
Multiplying \eqref{eq:PPeps} by $\vfi\chi_n$ and integrating over $(0,\infty)\times(0,1)$ we get
\begin{equation}
\label{eq:def-hyp-eps.3}
\begin{split}
\int_0^\infty\!\!\int_0^1 &\left(\rek\pt\vfi\chi_n+\uek\rek\left(\px\vfi\chi_n+\vfi\chi_n'\right)\right)dtdx\\
&-\eps_k\int_0^\infty\!\!\int_0^1\px\rek\left(\px\vfi\chi_n+\vfi\chi_n'\right) dtdx\\
&-\int_0^\infty\left(\uek(t,1)\rho_r(t)\vfi(t,1)-\uek(t,0)\rho_l(t)\vfi(t,0)\right)dt\\
&+\eps_k\int_0^\infty\left(\px\rek(t,1)\vfi(t,1)-\px\rek(t,0)\vfi(t,0)\right)dt\\
&+\int_0^1\rho_0(x)\vfi(0,x)\chi_n(x)dx=0.
\end{split}
\end{equation}
As $k\to\infty$, thanks to  Lemmas \ref{lm:uBV},  \ref{lm:rcomp}, and \ref{lm:ucomp}, we get
\begin{equation}
\label{eq:def-hyp-eps.4}
\begin{split}
\lim\limits_k\eps_k&\int_0^\infty\left(\px\rek(t,1)\vfi(t,1)-\px\rek(t,0)\vfi(t,0)\right)dt\\
=&-\int_0^\infty\!\!\int_0^1 \left(\rho\pt\vfi\chi_n+u\rho\left(\px\vfi\chi_n+\vfi\chi_n'\right)\right)dtdx\\
&+\int_0^\infty\left(u(t,1)\rho_r(t)\vfi(t,1)-u(t,0)\rho_l(t)\vfi(t,0)\right)dt\\
&-\int_0^1\rho_0(x)\vfi(0,x)\chi_n(x)dx.
\end{split}
\end{equation}
As $n\to\infty$ we get \eqref{eq:def-hyp-eps.2}.

Using \eqref{eq:def-hyp-eps.2} in \eqref{eq:def-hyp-eps.1} we get \eqref{eq:def-hyp}.
\end{proof}

\section{Proof of Theorem \ref{th:main02}}
\label{stationaryproblem}

The claimed convergences (\ref{stab01})-(\ref{stab04}) relies on 
\begin{align*}
\rho&\in L^\infty((0,\infty)\times(0,1)),\\
 u&\in L^\infty(0,\infty;W^{1,\infty}(0,1)),\\
 p&\in L^\infty(0,\infty;W^{1,\infty}(0,1)).
\end{align*}
Indeed, we have that
\begin{align*}
\{\rho(t,\cdot)\}_{t\ge0}&\>\text{is uniformly bounded in}\> L^\infty(0,1),\\
\{ u(t,\cdot)\}_{t\ge0}&\>\text{is uniformly bounded in}\> W^{1,\infty}(0,1),\\
\{ p(t,\cdot)\}_{t\ge0}&\>\text{is uniformly bounded in}\> W^{1,\infty}(0,1).
\end{align*}

Using  the strong convergence of $u(t,\cdot) \rightarrow u_\infty$  we can pass to the limit in the relation (\ref{eq:defu})  obtaining  
\begin{equation}
\label{eq:end1}
\px p_{\infty}= - \alpha \rho_\infty G(u_\infty)
\end{equation}
  and thus  
$$
u_\infty = g\left(\frac{\px p_\infty}{\alpha \rho_\infty}\right).
$$
This suffices to pass to limit in the left hand side of (\ref{eq:def-ell}).
But it is not sufficient to  prove a convergence of 
$F = F(t,\cdot,\rho(t,\cdot))$ to $F_\infty(\cdot,\rho_\infty)$ as $t \rightarrow \infty$. 

The additional assumption
$$
\px p\in L^{\infty}(0,\infty;BV(0,1))
$$
allows us to improve the weak convergence in \eqref{stab04} to the strong one in \eqref{eq:stab05}. Indeed, we have that 
\begin{align*}
\{ p(t,\cdot)\}_{t\ge0}&\>\text{is uniformly bounded in}\> W^{1,\infty}(0,1)\\
\{\px p(t,\cdot)\}_{t\ge0}&\>\text{is uniformly bounded in}\> BV(0,1).
\end{align*}

Then with (\ref{eq:defu}) we can conclude strong convergence of $\rho$ 
$$
\rho(t,\cdot) \rightarrow  \rho_\infty  \quad \mbox{ in } L^r(0,1) , \, 1 \leq r<\infty \mbox{ as } t \rightarrow \infty.
$$  
This allows us to pass to the limit in the nonlinearity $f$ and $F$  and thus in the  the complete weak formulations (\ref{eq:def-ell}) and (\ref{eq:def-hyp}) of the model. 
This - as a byproduct - gives the existence of a solution
\begin{equation}
	\label{eq:sollongtime}
(\rho_\infty,\,p_\infty) \in L^\infty(0,1) \times W^{1,\infty}(0,1)
\end{equation}
of the stationary problem (\ref{eq:CLlongtime}) in the sense of Definition \ref{def:sol}.
In particular, $\px p_\infty\in L^\infty(0,1)$ due to \eqref{eq:end1}.

Obviously, there is no granted uniqueness. 

We gain the uniqueness of the solutions of the stationary  problem working directly on (\ref{eq:CLlongtime}).
We rewrite the problem as follows omitting the the index $\infty$ for the unknown quantities in order to keep the notation simple
\begin{equation}
	\label{eq:PPlt}
	\begin{cases}
		\px(\rho u)=0,& \\
		\px p=-\alpha \rho u|u|,& \\
		\px u={\rho^{-2}}(f_{\infty}(x)-\beta_1T-\beta_2T^4),&\\
		T=\gamma-\rho.&
	\end{cases}
\end{equation}
 
The first equation in (\ref{eq:PPlt}) allows to define the flux $j=\rho u$ which has to be constant. If we assume no vacuum ($\rho>0$) and using $u=\frac{j}{\rho}$ we can rewrite the system as 
\begin{align}
			\px p =&-\alpha \frac{j|j|}{\rho},  \label{plt}\\ 
		j \px\rho =&-(f_{\infty}(x)-\beta_1 (\gamma - \rho)-\beta_2 (\gamma - \rho)^4). \label{rho_lt}
\end{align}
We continue our analysis reminding that our solution satisfies\eqref{eq:sollongtime}.

From a parabolic trough application point of view vacuum solutions are not relevant. Infact below we will comment on the non-occurrence of vacuum solutions.

Equation (\ref{plt}) allows us to relate the flux $j$ to the  boundary conditions on $p$
\begin{equation}
 p_{r,\infty}-p_{l,\infty} = p_{diff}= - \alpha \, j|j|  \int_0^1 \frac{1}{\rho(x)} dx
	\label{mvsDeltap}
\end{equation}
i.e. on 
\begin{equation}
	p_{r,\infty}-p_{l,\infty} < 0 \Leftrightarrow j > 0, \qquad
	p_{r,\infty}-p_{l,\infty} = 0 \Leftrightarrow j = 0, \qquad
	p_{r,\infty}-p_{l,\infty} > 0 \Leftrightarrow j < 0.
	\label{mvsDeltap1}
\end{equation}
The case $j=0$ (induced by $p_{r,\infty}-p_{l,\infty} = 0$)  gives a polynomial equation for $\rho$
\begin{equation}
	0 = f_{\infty}(x)-\beta_1 (\gamma - \rho(x))-\beta_2 (\gamma - \rho(x))^4),
	\label{mzero}
\end{equation}
which - given the special form $p(y) = y^4 + c_0 y - c_1 = 0, \quad c_0,c_1 >0$ - has a unique non-negative solution for $\rho$.   
Only if this (unique) solution of the equation  (\ref{mzero}) is compatible with the boundary conditions on $\rho$ we have a continuous solution for the stationary case for $j=0$. 


For the cases $j \ne 0$ we have the ODE (\ref{rho_lt})  for $\rho$ with 
- according to the inflow conditions - an initial value $\rho_{l,\infty}$ on the left for $j>0$ and an initial value $\rho_{r,\infty}$ on the right for $j<0$, respectively.
Under the assumptions of Lemma \ref{lm:maxp} vacuum does not occur. This can be shown adapting Lemma \ref{lm:maxp} to the stationary case and using the  same technique (known sub and super-solutions) to proof it. 

From now on we discuss the case  $j >0$, the discussion for $j<0$ is
analogous. 
We combine the two equations (\ref{plt}) and (\ref{rho_lt}), use the formulation
\begin{equation}	
	\rho(x) = \rho_{l,\infty} + \int_0^x\px\rho(y)dy
	= \rho_{l,\infty} - \frac{1}{j} \int_0^x 
	\left(
	f_{\infty}(y)-\beta_1 (\gamma - \rho(y))-\beta_2 (\gamma - \rho(y))^4
	\right) dy
	\label{comb_lt}
\end{equation}
and obtain
\begin{align}
 	p_{diff}  = &   -\alpha \, j|j|  \int_0^1 \frac{1}{\rho(x)} dx \nonumber \\
 	 = & -\alpha \, j|j|  \int_0^1 
 	\frac{1}{\rho_{l,\infty} - \frac{1}{j} \int_0^x 
 		\left(
 		f_{\infty}(y)-\beta_1 (\gamma - \rho(y))-\beta_2 (\gamma - \rho(y))^4
 			\right) dy} dx.
 	\label{comb_lt1}
 \end{align}
We see that  for
\begin{equation}	
	\px \rho =  - \frac{1}{j} \int_0^x 
	\left(
	f_{\infty}(y)-\beta_1 (\gamma - \rho(y))-\beta_2 (\gamma - \rho(y))^4
	\right) dy   \geq 0  
	\label{comb_lt2}
\end{equation}
the quantity $p_{diff}$ strictly decreases as $j$ increases.


For $\px\rho  < 0$ we derive $p_{diff}$ with respect to $j$ (denoting $\frac{d}{dj} = ()^{\prime}$) and obtain 
\begin{eqnarray}
	p_{diff}^{\prime} 
	& = & - 2\alpha \,|j|  \int_0^1 
	\frac{1}{\rho_{l,\infty} - \frac{1}{j} \int_0^x 
		\left(
		f_{\infty}(y)-\beta_1 (\gamma - \rho(y))-\beta_2 (\gamma - \rho(y))^4
		\right) dy} dx
	\nonumber \\
	&  & -\alpha \int_0^1 
    \frac{\int_0^x 
    	\left(
    	f_{\infty}(y)-\beta_1 (\gamma - \rho(y))-\beta_2 (\gamma - \rho(y))^4
    	\right) dy
    }{\left[\rho_{l,\infty} - \frac{1}{j} \int_0^x 
	\left(
	f_{\infty}(y)-\beta_1 (\gamma - \rho(y))-\beta_2 (\gamma - \rho(y))^4
	\right) dy\right]^2} dx   < 0. 
	\label{comb_lt3}
\end{eqnarray}
Note that the last inequality is only true for $\px\rho  < 0$.
Thus, in both cases,  $\px\rho  \geq  0$ and $\px\rho  < 0$, the quantity  $p_{diff}$ is a strictly monoton decreasing function in $j$ and therefore strictly monoton decreasing everywhere. 

In addition we have  $p_{diff} (0) = 0$ and  $p_{diff} (j) \rightarrow  -\infty$  as $j \rightarrow \infty $.
Since $p_{r,\infty}-p_{l,\infty} < 0$ for $j>0$ there exists a unique $j>0$ and therefore a unique solutions of the problem. 

As a consequence we have the convergence along all the times $t\to\infty$ stated in \eqref{eq:stab05} and the proof of Theorem \ref{th:main02} is concluded.


%
%
%
%
\section{Conclusion}
\label{sec:conclusion}
We present an analysis of an asymptotic model for the thermo-fluid dynamics 
in a single collector pipe of a parabolic trough power plant. The analysis covers existence, the stationary problem and { (partial) results on the long time behaviour}. 
The results obtained are the basis for further investigations, in particular for the analysis of a network of collector pipes in a realistic PPTP. There 
a deep analytical understanding can help to improve the simulation and optimisation approaches. Insofar this paper present a important step towards the analytical understanding of such a complex thermo-fluid dynamic setting
with a direct application in reality. \\

%
%
\appendix
\section{A comment on the validity of assumption (\ref{eq:ass3})}
\label{assumption}
Here we comment on the assumption (\ref{eq:ass3}) on the function $f$. 
The question is if this assumption is giustified in real world PPTP applications. To answer this, we take data from an existing parabolic trough power plant - NOOR I in Marocco - which was already studied in \cite{BGSP}.
We start with 
\begin{equation}
	f(t,x) = q(t,x) + \beta_1 (\gamma - \rho_{out}) + \beta_2 (\gamma - \rho_{sky})^4,
	\label{estimate01}
\end{equation}
were $0 \leq q$ is the heat source due to the solar radiation.
Since $\gamma$  is an upper bound for the densities we see that the first inequality  $0 \leq f$ is always tue. \\
For the second inequality  
\begin{equation}
	f \leq  \beta_1 (\gamma - \gamma_{\star}) + \beta_2 (\gamma - \gamma_{\star})^4
	\label{estimate02}
\end{equation}
we consider the unscaled version. We denote the original unscaled quantities with a $\tilde{ }$   (see \cite{BGSP}).
We have to show  
\begin{equation}
	\tilde{f} = 
	\tilde{q} 
	+\tilde{\beta}_1 (\tilde{\gamma} - \tilde{\rho}_{out}) + \tilde{\beta}_2 (\tilde{\gamma} - \tilde{\rho}_{sky})^4	
	\leq  \tilde{\beta}_1 (\tilde{\gamma} - \tilde{\gamma}_{\star}) + \tilde{\beta}_2 (\tilde{\gamma} - \tilde{\gamma}_{\star})^4
	\label{estimate02}
\end{equation}
and thus 
\begin{equation}
	\tilde{q}  \leq
	\tilde{\beta}_1 (\tilde{\gamma} - \tilde{\gamma}_{\star})
	- \tilde{\beta}_1 (\tilde{\gamma} - \tilde{\rho}_{out})
	+ \tilde{\beta}_2 (\tilde{\gamma} - \tilde{\gamma}_{\star})^4
	- \tilde{\beta}_2 (\tilde{\gamma} - \tilde{\rho}_{sky})^4	.
	\label{estimate03}
\end{equation}
Using the same scaling as in \cite{BGSP} we can compare physical values.
Considering i.e. a very low density $\tilde{\gamma}_{\star} = 0.05 \rho_{out}$ -- which corresponds to a very high temperature of $\tilde{T}_{\star} = 1360 \,K$ far from the operational domain --
we obtain values for the various powers per lenght (along the collector tube)
\begin{align*}
	\tilde{q}  = & \, 4300 \,W/m, \\
	\tilde{\beta}_1 (\tilde{\gamma} - \tilde{\gamma}_{\star})
	- \tilde{\beta}_1 (\tilde{\gamma} - \tilde{\rho}_{out}) 
	 = & \, 8904 \, W/m,\\
	\tilde{\beta}_2 (\tilde{\gamma} - \tilde{\gamma}_{\star})^4
	- \tilde{\beta}_2 (\tilde{\gamma} - \tilde{\rho}_{sky})^4	
	 = & \, 6092 \, W/m.
\end{align*}
Thus we see that the second inequality is easily fullfilled since
\begin{equation}
	4300 \, W/m  <  8904 \, W/m + 6092 \, W/m.
	\label{estimate05}
\end{equation}
Our assumption could only become critical in the future for more advanced PTPP's if the absorbed energy $q$ becomes significantely larger (by increassing the mirrors per tube) or/and if
the losses  (convection and radiation) become significantely less due to better isolation or protection against radiative losses.
However, there is still a considerable margin in the assumption.

\section*{Acknowledgments}

   G.\ M.\ Coclite is member of Gruppo Nazionale per l'Analisi Matematica, la Probabilit\`a e le loro Applicazioni (GNAMPA) of the Istituto Nazionale di Alta Matematica (INdAM).
 
   G.\ M.\ Coclite has been partially supported by the Research Project of National Relevance ``Evolution problems involving interacting scales'' granted by the Italian Ministry of Education, University and Research (MUR Prin 2022, project code 2022M9BKBC, Grant No. CUP D53D23005880006).
 
   G.\ M.\ Coclite acknowledges financial support under the National Recovery and Resilience Plan (NRRP) funded by the European Union - NextGenerationEU -  
Project Title ``Mathematical Modeling of Biodiversity in the Mediterranean sea: from bacteria to predators, from meadows to currents'' - project code P202254HT8 - CUP B53D23027760001 -  
Grant Assignment Decree No. 1379 adopted on 01/09/2023 by the Italian Ministry of University and Research (MUR).

G.\ M.\ Coclite  has
been partially supported by the Project funded under the National Recovery and
Resilience Plan (NRRP), Mission 4 Component 2 Investment 1.4 -Call for tender
No. 3138 of 16/12/2021 of Italian Ministry of University and Research funded by
the European Union -NextGenerationEUoAward Number: CN000023, Concession
Decree No. 1033 of 17/06/2022 adopted by the Italian Ministry of University and
Research, CUP: D93C22000410001, Centro Nazionale per la Mobilit\`a Sostenibile.

G.\ M.\ Coclite  was partially supported by the Italian Ministry of University and Research under the Programme ``Department of Excellence'' Legge 232/2016 (Grant No. CUP - D93C23000100001).

G.\ Gasser  acknowledges the support by the Deutsche Forschungsgemeinschaft (DFG) within the Research Training Group GRK 2583 ``Modeling, Simulation and Optimization of Fluid Dynamic Applications".

G.\ M.\ Coclite and G.\ Gasser express their gratitude to HIAS (Hamburg Institute for Advanced Study) for their warm hospitality.

\end{document}